\documentclass{amsart}
\usepackage{amsfonts,amsmath,amssymb}

\input xy
\xyoption{all}

\newcommand{\M}{\mathcal M}

\newcommand{\w}{\omega}
\newcommand{\IR}{\mathbb R}

\newcommand{\I}{\mathcal I}
\newcommand{\J}{\mathcal J}
\newcommand{\A}{\mathcal A}

\newcommand{\B}{\mathcal B}
\newcommand{\X}{\mathcal X}
\newcommand{\F}{\mathcal F}
\newcommand{\K}{\mathcal K}
\newcommand{\C}{\mathcal C}
\newcommand{\U}{\mathcal U}
\newcommand{\V}{\mathcal V}

\newcommand{\HH}{\mathcal H}
\newcommand{\II}{\mathbb I}

\newcommand{\IN}{\mathbb N}
\newcommand{\IQ}{\mathbb Q}

\newcommand{\add}{\mathrm{add}}
\newcommand{\cov}{\mathrm{cov}}
\newcommand{\non}{\mathrm{non}}
\newcommand{\cof}{\mathrm{cof}}

\newcommand{\e}{\varepsilon}
\newcommand{\Ra}{\Rightarrow}

\newcommand{\oD}{\overline{\mathcal D}}
\newcommand{\Z}{\mathcal Z}

\newtheorem{theorem}{Theorem}[section]
\newtheorem{lemma}[theorem]{Lemma}

\newtheorem{corollary}[theorem]{Corollary}

\newtheorem{problem}[theorem]{Problem}
\newtheorem{example}[theorem]{Example}

\newtheorem{proposition}[theorem]{Proposition}

\theoremstyle{definition}

\title{Topologically invariant $\sigma$-ideals on Euclidean spaces}
\author{Taras Banakh, Micha\l\ Morayne, Robert Ra\l owski, Szymon \.Zeberski}
\address[Taras Banakh]{Department of Mathematics, Ivan Franko National University of Lviv, Ukraine, and
Institute of Mathematics, Jan Kochanowski University, Kielce, Poland}
\email{t.o.banakh@gmail.com}
\address[Micha\l\ Morayne, Robert Ra\l owski, Szymon \. Zeberski]{Institute of Mathematics and Computer Science, Wroc\l aw University of Technology, Wroc\l aw, Poland}
\email{michal.morayne@pwr.wroc.pl, robert.ralowski@pwr.wroc.pl, szymon.zeberski@pwr.wroc.pl}

\begin{document}
\begin{abstract} We study and classify topologically invariant $\sigma$-ideals with an analytic base on Euclidean spaces and evaluate the cardinal characteristics of such ideals.
\end{abstract}
\maketitle
\footnotetext{\noindent AMS classification: 03E17,03E75,57N15\\
Keywords: topologically invariant ideal, cardinal characteristics, meager set, Cantor set; \\
The work has been partially financed by NCN means granted by decision DEC-2011/01/B/ST1/01439. }

\section{Introduction}
The $\sigma$-ideals of Lebesgue's measure zero sets and meager sets have been the subject of extensive research devoted to revealing the fine structure of the real line and, more generally, the Euclidean spaces. This research resulted in finding the relations between the most important cardinal characteristics of these two $\sigma$-ideals. These relations are described by the Cicho\'n diagram (see e.g. \cite{fremlin}, \cite{bartoszynskijudah}). Both ideals have Borel base and differ by the property that the ideal $\M$ of meager sets in topologically invariant while the ideal $\mathcal N$ of Lebesgue null sets is not.

In this paper we examine the properties  of non-trivial
topologically invariant $\sigma$-ideals with Borel base on Euclidean spaces $\IR^n$. In particular, we show that the $\sigma$-ideal of meager sets, $\mathcal{M}$, is the biggest topologically invariant $\sigma$-ideal with Borel base on $\IR^n$ while the $\sigma$-ideal generated by the so called tame Cantor sets, $\sigma\C_0$, is the smallest one. Our main results concern the four cardinal characteristics of these two $\sigma$-ideals: the additivity ($\add$), the uniformity ($\non$), the covering ($\cov$) and the cofinality ($\cof$).
In fact, we show that the uniformity and the covering numbers are the same for all non-trivial topologically invariant $\sigma$-ideals with Borel base on Euclidean spaces and the remaining two cardinals may be different than the corresponding characteristics of the ideal $\M$. Yet,  the respective cardinal characteristics of the extremal ideals $\sigma\C_0$ and $\mathcal{M}$ coincide.
The same concerns the $\sigma$-ideals $\sigma\oD_k$ generated by closed subset of dimension $k<n$ in $\IR^n$.

Properties of topologically invariant $\sigma$-ideals may be different in different topological spaces. There are other natural spaces where it would be interesting and useful to know these properties. The Hilbert cube case is examined in \cite{hilbert}.

\section{Definitions, notation and statement of principal results}

The symbols $\IR$, $\IQ$, and $\omega$ will have the usual meaning, i.e. they will denote the real line, the set of rational numbers, and the set of finite ordinals (i.e., the set of nonnegative integers), respectively. A {\em Euclidean space} is a topological space homeomorphic to the space $\IR^n$ for some positive integer $n$. All topological spaces considered in this paper are assumed to be separable and metrizable.

Let us recall that a subset $A$ of a topological space $X$ is {\em analytic} if $A$ is the  image of a Polish space under a continuous map. A subset $A\subseteq X$ has the {\em Baire property} (briefly, $A$ is a {\em BP-set}) if there is an open subset $U\subseteq X$ such that the symmetric difference $A\triangle U=(A\setminus U)\cup(U\setminus A)$ is meager in $X$.

A non-empty family $\I$ of subsets of a set $X$ is called an {\it ideal on $X$} if $\I$ is {\em hereditary} (with respect to taking subsets) and {\em additive} in the sense that the union $A\cup B$ of any two sets $A,B\in\I$ belongs to $\I$. An ideal $\I$ is called a {\em $\sigma$-ideal} if the union $\bigcup\A$ of any countable family $\A\subseteq\I$ belongs to $\I$.
An ideal $\I$ on a set $X$ is {\em non-trivial} if $\I$ contains an uncountable set and $\I$ is not equal to the ideal $\mathcal{P}(X)$ of all subsets of $X$.

A subfamily $\B\subseteq\I$ is a {\em base} of an ideal $\I$ if each element $A\in\I$ is a subset of some set $B\in\B.$
We say that an ideal $\I$ of subsets of a topological space $X$ has {\it  Borel} (resp. {\it analytic}, {\em BP-}) {\it base}, or that $\I$ is an {\it ideal with Borel} (resp. {\it analytic}, {\em BP-}) {\it base}, if there exists a base for $\I$ consisting of Borel (analytic, BP-) subsets of $X$.

It is well-known that each Borel subset of a Polish space is analytic and each analytic subset of a metrizable separable space $X$ has the Baire property in $X$. This implies that for an ideal $\I$ on a  Polish space we have the following implications:
\smallskip

\centerline{$\I$ has Borel base $\Ra$ $\I$ has analytic base $\Ra$ $\I$ has BP-base.}
\medskip

A $\sigma$-ideal $\I$ on a topological space $X$ is {\it topologically invariant} if $\I$ is transformed onto $\I$ by any homeomorphism $h$ of $X$ i.e. $\I=\{h(A):A\in\I\}$.

It is clear that for each topological space $X$ the ideal $\M$ of meager subsets of $X$ is topologically invariant. It turns out that this ideal is the largest one among non-trivial topologically invariant $\sigma$-ideals with BP-base on a Euclidean space $X=\IR^n$.

\begin{theorem}\label{largest} Each non-trivial $\sigma$-ideal $\I$ with BP-base on a Euclidean space $X=\IR^n$ is contained in the ideal $\M$ of meager subsets of $X$.
\end{theorem}

\begin{proof} Let us assume that $\I\not\subseteq \M$. Let $A\in\I\setminus\M$. Since $\I$ has BP-base, we can assume that the non-meager set $A$ has the Baire property and hence contains a $G_\delta$-subset $G_U\subseteq A$, dense in some open subset $U$ of $X$. Since the Euclidean space $X$ is topologically homogeneous, we can choose a countable family $H$ of homeomorphisms of $X$ such that $\bigcup_{h\in H}h(U)$ is dense in $X$. Then the $G_{\delta\sigma}$-set $D=\bigcup_{h\in H}h(G_U)$ is comeager in $X$ and hence contains a subset $G\subseteq D$, which is dense $G_\delta$ in $X$. By the topological invariance of $\I$, the set $D$ and its $G_\delta$-subset $G$ belong to the $\sigma$-ideal $\I$.

By \cite{CW}, for the dense $G_\delta$ subset $G$ of the Euclidean space $X=\IR^n$, there are homeomorphisms $h_0,\dots,h_n:X\to X$  such that $X=\bigcup_{k=0}^n h_k(G)$. Then $X$ belongs to $\I$ by the topological invariance of $\I$, which means that the ideal $\I$ is trivial.
\end{proof}

By Theorem~\ref{largest}, $\M$ is the largest non-trivial $\sigma$-ideal with Borel base on $\IR^n$. Now we describe the smallest non-trivial $\sigma$-ideal with Borel base on $\IR^n$. It is denoted by $\sigma\C_0$ and is generated by tame Cantor sets in $\IR^n$.

A subset $C$ of a Polish space $X$ is called a {\em Cantor set\/} if $C$ is homeomorphic to the Cantor cube $\{0,1\}^\w$. By Brouwer's characterization~\cite[7.4]{Ke}, a subset $C\subseteq X$ is a Cantor set if and only if it is compact, zero-dimensional and has no isolated points.

Two subsets $A,B$ of a topological space $X$ are called {\em ambiently homeomorphic} if $h(A)=B$ for some homeomorphism $h:X\to X$ of $X$.

A subset $C$ of a Euclidean space $X=\IR^n$ is called a {\em tame Cantor set} if
it is ambiently homeomorphic to a Cantor set contained in the line $\IR\times\{0\}^{n-1}\subseteq\IR^n$.
Since any two Cantor sets on the real line are ambiently homeomorphic, any two tame Cantor sets in $\IR^n$ are ambiently homeomorphic.

By \cite{Osborn}, a closed subset $C\subseteq \IR^n$ is a tame Cantor set if and only if for each $\e>0$ the set $C$ is contained in the union $\bigcup\F$ of a finite family $\F$ of pairwise disjoint open sets homeomorphic to $\IR^n$ and having diameter $<\e$. This characterization implies that {\em each Cantor set in $\IR^n$ contains a tame Cantor set}.

It is known \cite{Bing} that for $n\le 2$ each Cantor set in $\IR^n$ is tame while for $n\ge 3$ a Cantor subset $C\subseteq \IR^n$ is tame if and only if $C$ is a {\em $Z_2$-set}  in $\IR^n$. The latter means that each map $f:[0,1]^2\to\IR^n$ can be uniformly approximated by a map $f':[0,1]^2\to\IR^n\setminus C$. Cantor sets which are not tame are called {\em wild}, see \cite{An}, \cite{Blank}, \cite{Sher}.

For a Euclidean space $X=\IR^n$ by $\sigma\C_0$ we denote the $\sigma$-ideal generated by  tame Cantor sets in $X$. It consists of all subsets of countable unions of tame Cantor sets in $X$.

\begin{theorem}\label{smallest} The $\sigma$-ideal $\sigma\C_0$ is contained in each non-trivial $\sigma$-ideal $\I$ with analytic base on $X=\IR^n$.
\end{theorem}

\begin{proof} The ideal $\I$, being non-trivial, contains an uncountable set $A$. Since $\I$ has analytic base, we can assume that the uncountable set $A$ is analytic and hence contains a Cantor set $C$ according to Souslin's Theorem \cite[29.1]{Ke}. Since each Cantor set in $\IR^n$ contains a tame Cantor set, we can assume that $C$ is a tame Cantor set in $X$. So, the ideal $\I$ contains a tame Cantor set. Since any two tame Cantor sets in $X=\IR^n$ are ambiently homeomorphic, by the topological invariance, the ideal $\I$ contains all tame Cantor sets and being a $\sigma$-ideal, contains the $\sigma$-ideal $\sigma\C_0$ generated by tame Cantor sets in $X$.
\end{proof}

\begin{corollary}\label{minimax} If $\mathcal I$ is a non-trivial topologically invariant $\sigma$-ideal $\I$ with analytic base on a Euclidean space $X=\IR^n$, then $\sigma\C_0\subseteq\I\subseteq \M$.
\end{corollary}

This corollary will be used to evaluate the cardinal characteristics of non-trivial topologically invariant $\sigma$-ideals with Borel base on Euclidean spaces.

Given an ideal $\I$ on a set $X=\bigcup\I$, we shall consider the following four cardinal characteristics of $\I$:
$$\begin{aligned}
&\add(\I)=\min\{|\A|:\A\subseteq\I,\;\;\textstyle{\bigcup}\A\notin\I\},\\
&\non(\I)=\min\{|A|:A\subseteq X,\;\;A\notin\I\},\\
&\cov(\I)=\min\{|\A|:\A\subseteq\I,\;\textstyle{\bigcup}\A=X\},\\
&\cof(\I)=\min\{|\A|:\A\subseteq\I\;\;\forall B\in\I\;\;\exists A\in\A\;\;(B\subseteq A)\}.
\end{aligned}
$$

In fact, these four cardinal characteristics can be expressed using the following two cardinal characteristics defined for any pair $\I\subseteq\J$ of ideals:
$$\begin{aligned}
&\add(\I,\J)=\min\{|\A|:\A\subseteq\I,\;\;\textstyle{\bigcup}\A\notin\J\} \mbox{ \ and}\\
&\cof(\I,\J)=\min\{|\A|:\A\subseteq\J\;\;\forall B\in\I\;\exists A\in\A\;\;(B\subseteq A)\}.
\end{aligned}
$$
Namely,
$$\add(\I)=\add(\I,\I),\;\;\non(\I)=\add(\F,\I),\;\;
\cov(\I)=\cof(\F,\I),\;\;\cof(\I)=\cof(\I,\I)$$where $\F$ stands for the ideal of finite subsets of $X$.

The cardinal characteristics of the largest $\sigma$-ideal $\M$ have been thoroughly studied, see \cite{bartoszynskijudah}. The (relative) cardinal characteristics of the smallest $\sigma$-ideal $\sigma\C_0$ (in $\M$) are evaluated in the following theorem which will be proved in Section~\ref{s:main}. Theorem \ref{main} and subsequent Corollary \ref{any-ideal} are the principal results of this article.

\begin{theorem}\label{main} For the $\sigma$-ideal $\sigma\C_0$ on a Euclidean space $X=\IR^n$  the following equalities hold:
\begin{enumerate}
\item $\cov(\sigma\C_0)=\cov(\M)$;
\item $\non(\sigma\C_0)=\non(\M)$;
\item $\add(\sigma\C_0)=\add(\sigma\C_0,\M)=\add(\M)$;
\item $\cof(\sigma\C_0)=\cof(\sigma\C_0,\M)=\cof(\M)$.
\end{enumerate}
\end{theorem}

Corollary~\ref{minimax} and Theorem~\ref{main} imply the following corollary.

\begin{corollary}\label{any-ideal} For any non-trivial topologically invariant $\sigma$-ideal $\I$ with analytic base on a Euclidean space $X=\IR^n$ we get:
\begin{enumerate}
\item $\cov(\I)=\cov(\M)$;
\item $\non(\I)=\non(\M)$;
\item $\add(\I)\le\add(\M)$;
\item $\cof(\I)\ge\cof(\M)$.
\end{enumerate}
\end{corollary}

Thus, for $X=\IR^n$, the following variant of Cicho\'n's diagram describes relations between cardinal characteristics of the ideal $\M$ and any non-trivial topologically invariant $\sigma$-ideal $\I$ ($a\to b$ stands for  $a\le b$):

$$
\xymatrix{
&\non(\I)\ar@{=}[r]&\non(\M)\ar[r]&\cof(\M)\ar[r]&\cof(\I)\ar[r]&\mathfrak c\\
\w_1\ar[r]&\add(\I)\ar[u]\ar[r]&\add(\M)\ar[u]\ar[r]
&\cov(\M)\ar[u]\ar@{=}[r]&\cov(\I)\ar[u]&
}
$$
\smallskip

The following example shows that the inequalities $\add(\I)\le\add(\M)$ and $\cof(\M)\le\cof(\I)$ can be strict.

Below for a subset $A$ of a Polish space $X$ by $\I_A$ we denote the smallest topologically invariant $\sigma$-ideal containing the set $A$. It consists of all subsets of countable unions $\bigcup_{n\in\w}^\infty h_n(A)$ where $h_n:X\to X$, $n\in\w$, are homeomorphisms of $X$.

\begin{example}  The  $\sigma$-ideal $\I_\II\subseteq\mathcal{P}(\IR^2)$ generated by the interval $\II=[0,1]\times\{0\}$ in the plane $\mathbb R^2$ has cardinal characteristics:
$$\add(\I_\II)=\w_1,\;\;\non(\I_\II)=\non(\M),\;\;\cov(\I_\II)=\cov(\M),
\;\mbox{ and }\;\cof(\I_\II)=\mathfrak c.$$
\end{example}

\begin{proof}
The equalities $\cov(\I_\II)=\cov(\M)$ and $\non(\I_\II)=\non(\M)$ follow from Corollary \ref{any-ideal}.

The equality $\add(\I_\II)=\w_1$ will follow if we check that $\bigcup_{t\in T}[0,1]\times\{t\}\notin\I_\II$ for any uncountable subset $T\subseteq [0,1]$. Assuming the opposite, we can find a homeomorphism $h:\IR^2\rightarrow\IR^2$ such that the set
 $$\{t\in T:\ h(\II)\cap \big([0,1]\times\{t\}\big) \text{ contains a line segment}\}
 $$
 is uncountable. It yields an uncountable family of pairwise disjoint proper intervals in $[0,1]$ which is not possible.

To show that $\cof(\I_\II)=\mathfrak c$, choose any subfamily
 $\B\subseteq\I_\II$ of cardinality $|\mathcal B|=\cof(\I_\II)$ such that each set $A\in\I_\II$ is  contained in some set $B\in\mathcal B$. Let $\X=\{ [0,1]\times\{x\}:\ x\in\IR\}$. Notice that every member of $\B$ contains at most countably many members of the family $\X.$ This implies that $|\B|=\mathfrak c$.
\end{proof}

Corollary~\ref{any-ideal} will be applied to calculate the cardinal characteristics of the $\sigma$-ideal
$\sigma\oD_k$ generated by closed subsets of dimension $\le k$ in the Euclidean space $\IR^n$.
By \cite[1.8.11]{En2}, the ideal $\sigma\oD_{n-1}$ coincides with the ideal $\M$ of meager subsets of $\IR^n$.
The following theorem will be proved in Section~\ref{s:tD}.

\begin{theorem}\label{tD} For every number $0\le k<n$ the $\sigma$-ideal $\sigma\oD_k$ generated by closed at most $k$-dimensional
subsets of $\IR^n$ has cardinal characteristics:
$$\add(\sigma\oD_k)=\add(\M),\;\;
\cov(\sigma\oD_k)=\cov(\M),\;\;\non(\sigma\oD_k)=\non(\M),\;\;\cof(\sigma\oD_k)=\cof(\M).
$$
\end{theorem}

We finish this introductory section with two open problems.
A topologically invariant $\sigma$-ideal $\I$ will be called {\em 1-generated} if $\I=\I_A$ for some subset
%$A\subseteq X=\bigcup\I$
$A\in\mathcal{I}$.
Observe that  the $\sigma$-ideals $\sigma\C_0$ and $\M$ on $X=\IR^n$ are 1-generated: the $\sigma$-ideal $\sigma\C_0$ is generated by any tame Cantor set in $X$, while $\M$ is generated by the generalized Menger cube $M^n_{n-1}$, see \cite{Menger}, \cite[p.128]{Chi}.

\begin{problem} What are the cardinal characteristics of a 1-generated topologically invariant $\sigma$-ideal $\I_A$ with Borel base on a Euclidean space $X=\IR^n$. Is it true that $\add(\I)\in\{\w_1,\add(\M)\}$ and $\cof(\I)\in\{\cof(\M),\mathfrak c\}$ for any such ideal $\I$?
\end{problem}

Corollary~\ref{minimax} implies that $\M=\sigma\C_0$ is the unique topologically invariant $\sigma$-ideal with analytic base on the real line $\IR^1$. For higher-dimensional Euclidean spaces the ideals $\sigma\C_0$ and $\M$ are distinct.

\begin{problem}[M. Sabok] What is the cardinality  of the family of all topologically invariant $\sigma$-ideals with Borel base on a Euclidean space $X=\IR^n$ for $n\ge 2$. Is this cardinality equal to $2^{\mathfrak c}$?
\end{problem}

\section{Some properties of tame Cantor sets in Euclidean spaces}

In this section we shall establish some auxiliary facts related to tame Cantor sets and homeomorphism groups of Euclidean spaces. These facts will be used in the proof of Theorem~\ref{main}.

Corollary 2 of \cite{McM} or the characterization \cite{Osborn} of tame Cantor sets imply:

\begin{lemma}\label{l3.1} For any Cantor sets $C_1,\dots,C_n$ in the real line, the product $\prod_{i=1}^nC_i$ is a tame Cantor set in $\IR^n$.
\end{lemma}

\begin{lemma}\label{l3.2} For each $\e>0$, each compact nowhere dense subset $C\subseteq\IR$ and each open dense subset $U\subseteq \IR$ there is a homeomorphism $h:\IR\to\IR$ such that $h(C)\subseteq U$ and $\sup_{x\in\IR}|h(x)-x|<\e$.
\end{lemma}

\begin{proof} Since the set $C$ is compact and nowhere dense in $\IR$, we can choose an increasing sequence of real numbers
$$\alpha_0<\alpha_1<\ldots<\alpha_{2r}<\alpha_{2r+1}$$such that $C\subseteq \bigcup_{i=1}^r[\alpha_{2i-1},\alpha_{2i}]$ and  $\alpha_{i+1}-\alpha_i<\e$ for all $0\leq i\le 2r$.

Since $U$ is open and dense in $\IR$, for each $1\le i\le r$ we can choose two real numbers $\beta_{2i-1}<\beta_{2i}$ such that $[\beta_{2i-1},\beta_{2i}]\subseteq[\alpha_{2i-1},\alpha_{2i}]\cap U$.
Let $h:\IR\to\IR$ be the piecewise linear homeomorphism defined by the conditions:
\begin{itemize}
\item $h(x)=x$ for $x\in \IR\setminus(\alpha_0,\alpha_{2r+1})$,
\item $h(\alpha_i)=\beta_i$ for $0<i\le 2r$,
\item $h$ is linear on each segment $[\alpha_{i-1},\alpha_{i}]$, $0< i\le 2r+1$.
\end{itemize}
We have
$$\sup_{x\in\IR}|h(x)-x|\le\max_{1\le i\le 2r}|\beta_i-\alpha_i|\le\max_{1\le i\le r}|\alpha_{2i}-\alpha_{2i-1}|<\e$$ and
$$h(C)\subseteq \bigcup_{i=1}^r h([\alpha_{2i-1},\alpha_{2i}])=\bigcup_{i=1}^r[\beta_{2i-1},\beta_{2i}]\subseteq U.$$
\end{proof}

In the following lemma by $\|\cdot\|$ we denote the Euclidean norm on $\IR^n$.

\begin{lemma}\label{l3.3} For each $\e>0$, each compact nowhere dense subset $C\subseteq\IR\times\{0\}^{n-1}$ and each dense open set $U\subseteq\IR^n$ there is a homeomorphism $h:\IR^n\to\IR^n$ such that $h(C)\subseteq U$ and $\sup_{x\in\IR^n}\|h(x)-x\|<\e$.
\end{lemma}

\begin{proof} Let us represent the space $\IR^n$ as the product $\IR\times\IR^{n-1}$ and for each $\mathbf b\in\IR^{n-1}$ let  $i_{\mathbf b}:\IR\to \IR\times\{\mathbf b\}\subseteq\IR^n$ be the isometry $i_{\mathbf b}:x\mapsto (x,\mathbf b)$.
By $\mathbf 0$ we denote the zero vector of the space $\IR^{n-1}$.

Since $U$ is open and dense in $\IR^n=\IR\times\IR^{n-1}$, we can apply the Kuratowski-Ulam Theorem (see e.g. Corollary 1a sec. 22, V in \cite{kuratowski} or Theorem 8.41 in \cite{Ke}), and find a point $\mathbf b\in\IR^{n-1}$ such that $\|(0,\mathbf b)\|<\e/2$ and the set $U\cap(\IR\times\{\mathbf b\})$ is dense in $\IR\times\{\mathbf b\}$. Then the set $i^{-1}_{\mathbf b}(U)$ is open and dense in $\IR$. Since the set $i^{-1}_{\mathbf 0}(C)$ is compact and nowhere dense in the real line $\IR$, we can apply Lemma~\ref{l3.2} and find a homeomorphism $f:\IR\to\IR$ such that $f(i_{\mathbf 0}^{-1}(C))\subseteq i_{\mathbf b}^{-1}(U)$ and $\sup_{x\in\IR}|f(x)-x|<\e/2$.

The homeomorphism $f$ induces a homeomorphism $$h:\IR\times\IR^{n-1}\to\IR\times\IR^{n-1},\;h:(x,\mathbf y)\mapsto (f(x),\mathbf y+\mathbf b)$$
for which we have
$$h(C)=h(i_{\mathbf 0}^{-1}(C)\times\{\mathbf 0\})=f(i_{\mathbf 0}^{-1}(C))\times\{\mathbf{b}\}\subseteq
i_b^{-1}(U)\times\{\mathbf b\}=U\cap(\IR\times\{\mathbf b\})\subseteq U$$ and
$$\sup_{(x,\mathbf y)\in\IR\times\IR^{n-1}}\|h(x,\mathbf y)-(x,\mathbf y)\|=
\sup_{(x,\mathbf y)\in\IR\times\IR^{n-1}}\sqrt{|f(x)-x|^2+\|\mathbf b\|^2}\le
\sup_{x\in\IR}(|f(x)-x|+\|\mathbf b\|)<\frac\e2+\frac\e2=\e.$$
\end{proof}

By $\HH(\IR^n)$ we denote the homeomorphism group of the Euclidean space $\IR^n$, endowed with the compact-open topology. It can be identified with the closed subgroup of the homeomorphism group of the one-point compactification $\alpha\IR^n=\IR^n\cup\{\infty\}$ of $\IR^n$. This implies that $\HH(\IR^n)$ is a Polish group.

\begin{lemma}\label{l3.4} For each tame Cantor set $C\subseteq\IR^n$ and each open dense set $U\subseteq \IR^n$ the set
$$\HH^U_C=\{h\in\HH(\IR^n):h(C)\subseteq U\}$$is open and dense in $\HH(\IR^n)$.
\end{lemma}

\begin{proof} The openness of the set $\HH^U_C$ in $\HH(\IR^n)$ follows from the openness of the set $U$ and the definition of the compact-open topology on the group $\HH(\IR^n)$. It remains to prove that the set $\HH^U_C$ is dense in $\HH(\IR^n)$. Given a homeomorphism $h_0\in\HH(\IR^n)$, a compact set $K\subseteq\IR^n$ and $\e>0$, we need to find a homeomorphism $h\in\HH^U_C$ such that $\sup_{x\in K}\|h(x)-h_0(x)\|<\e$.

Let $C_0=h_0(C)\subseteq\IR^n$. Because the set $C_0$ is a tame Cantor set there exists a homeomorphism $f:\IR^n\to\IR^n$ such that $f(C_0)\subseteq\IR\times\{0\}^{n-1}$. Using the uniform continuity of the map $f^{-1}$ on each compact subset of $\IR^n$, we can find $\delta>0$, such that $\|f^{-1}(x)-f^{-1}(y)\|<\e$ for any points $x\in f\circ h_0(K)$ and $y\in\IR^n$ with $\|x-y\|<\delta$. By Lemma~\ref{l3.3}, there is a homeomorphism $g\in\HH(\IR^n)$ such that $g(f(C_0))\subseteq f(U)$ and $\sup_{x\in\IR^n}\|g(x)-x\|<\delta$.

Then the homeomorphism $h=f^{-1}\circ g\circ f\circ h_0$ has the required properties:
$$h(C)=f^{-1}\circ g\circ f\circ h_0(C)=f^{-1}\circ g\circ f(C_0)\subseteq f^{-1}(f(U))=U$$ and
$$\sup_{x\in K}\|h(x)-h_0(x)\|=\sup_{y\in h_0(K)}\|h\circ h_0^{-1}(y)-y\|=\sup_{y\in h_0(K)}\|f^{-1}\circ g\circ f(y)-f^{-1}\circ f(y)\|<\e.$$
\end{proof}

\begin{lemma}\label{l3.5} For each ($\sigma$-compact) set $A\in\sigma\C_0$ and each dense $G_\delta$-set $G\subseteq\IR^n$ the set
$$\HH^G_A=\{h\in\HH(\IR^n):h(A)\subseteq G\}$$
contains (is) a dense $G_\delta$-subset of $\HH(\IR^n)$.
\end{lemma}

\begin{proof} Let $A\in\sigma\C_0$. Then $A\subseteq\bigcup_{k\in\w}C_k$ for some tame Cantor sets $C_k\subseteq \IR^n$, $k\in\w$. The set $G$ is equal to the countable intersection $G=\bigcap_{k\in\w}U_k$ of a decreasing sequence of dense open sets $U_k\subseteq\IR^n$, $k\in\w$. It follows from Lemma~\ref{l3.4} that for any $i,j\in\w$ the set $\HH^{U_i}_{C_j}$ is a dense open set in the homeomorphism group $\HH(\IR^n)$. Then the intersection $\bigcap_{i,j\in\w}\HH^{U_i}_{C_j}\subseteq \HH^G_A$ is a dense $G_\delta$-subset of $\HH(\IR^n)$ contained in $\HH^G_A$.

If the set $A$ is $\sigma$-compact, then  $\HH_A^G$ is a $G_\delta$-set  in $\HH(\IR^n)$.
%This set is dense as it contains a dense $G_\delta$-set in $\HH(\IR^n)$.
\end{proof}

\section{Some known facts about cardinal characteristics of ideals}

In the proof of Theorem~\ref{main} we shall simultaneously work with $\sigma$-ideals on various topological spaces. To distinguish between such $\sigma$-ideals we shall use the following notations.

For a perfect Polish space $X$ by $\M(X)$ we denote the $\sigma$-ideal of all meager subsets of $X$, i.e. subsets of countable unions of closed sets with empty interior. It is well-known that the cardinal characteristics of the $\sigma$-ideal $\M(X)$ do not depend on a space $X$.

\begin{proposition}\label{equal}
If $X$ is a perfect Polish space, then
$$
\add(\M(X))=\add(\M(\w^\w)),\quad \cov(\M(X))=\cov(\M(\w^\w)),
$$
$$
\non(\M(X))=\non(\M(\w^\w)),\quad\cof(\M(X))=\cof(\M(\w^\w)).
$$
\end{proposition}

\begin{proof}
There is an embedding $\theta :\w^\w\to X$ whose image $\theta(\w^\w)$ is a dense $G_\delta$-subset of $X$ (it can be proved elementarily by a direct construction or follows from Theorem 2, sec. 36, IV and Theorem 3, sec. 36, II in \cite{kuratowski}). This gives the desired equalities.
\end{proof}

The above proposition justifies why we often use the symbol $\mathcal{M}$ without mentioning a specific space $X$.

For a topological space $X$ by $\sigma\K(X)$ we denote the $\sigma$-ideal generated by compact subsets of $X$ and put $\sigma\K=\sigma\K(\w^\w)$. It is known that
$$
\add(\sigma\K)=\non(\sigma\K)=\mathfrak b\;\;\mbox{ and }\;\;\cov(\sigma\K)=\cof(\sigma\K)=\mathfrak d,
$$
where $\mathfrak b$ (resp. $\mathfrak d$) is defined as the smallest cardinality $|B|$ of a subset $B\subset \w^\w$ which is {\em unbounded} (resp. {\em dominating}) in $\w^\w$ in the sense that for each $f\in\w^\w$ there is $g\in B$ such that $f\not\le^* g$ (resp. $f\le^* g$). Here for two functions $f,g\in\w^\w$ we write $f\le^* g$ if the set $\{n\in\w:f(n)>g(n)\}$ is finite.

We will use the following well known equalities (see e.g. \cite{bartoszynskijudah}).

\begin{lemma}\label{add-cof} For the ideals $\sigma\K$ and $\M$ on the Baire space $X=\w^\w$ the following equalities hold:
\begin{enumerate}
\item  $\add(\M)=\min\{\mathfrak b,\cov(\M)\}$ \textup{(Truss, Miller);}
\item $\cof(\M)=\max\{\mathfrak d,\non(\M)\}$ \textup{(Fremlin);}
\item $\cof(\sigma\K,\M)=\cof(\sigma\K)=\mathfrak d$ \textup{(Bartoszy\'nski)}.
\end{enumerate}
\end{lemma}

For a Euclidean space $X=\IR^n$ by $\sigma\C_0(\IR^n)$ we shall denote the $\sigma$-ideal generated by tame Cantor sets in $\IR^n$. Since each Cantor set in $\IR$ is tame, we get $\sigma\C_0(\IR)=\M(\IR)$.

Lemma~\ref{l3.1} and the fact that each meager set in $\IR$ is contained in a union of countably many Cantor sets imply the following lemma.

\begin{lemma}\label{l4.3} For any meager subsets $A_1,\dots,A_n\subseteq\IR$ the product $\prod_{k=1}^nA_k$ belongs to the ideal $\sigma\C_0(\IR^n)$.
\end{lemma}

This lemma yields another lemma.

\begin{lemma}\label{l4.4} For any zero-dimensional subspace $Z\subseteq\IR$ we get $\sigma\K(Z^n)\subseteq \sigma\C_0(\IR^n)$.
\end{lemma}

\section{Proof of Theorem~\ref{main}}\label{s:main}

Let $\M$ be the ideal of meager subsets of $\IR^n$ and $\sigma\C_0$ be the $\sigma$-ideal generated by the tame Cantor sets in $\IR^n$.
The proof of Theorem~\ref{main} is divided into four parts corresponding to the equalities (1)--(4) of Theorem \ref{main}.

\vspace{0.2 cm}

\noindent (1) $\cov(\sigma\C_0)=\cov(\M)$.

\vspace{0.2 cm}

The inequality $\cov(\M)\le\cov(\sigma\C_0)$ follows from the (trivial) inclusion $\sigma\C_0\subseteq\M$.

We will prove $\cov(\M)\ge\cov(\sigma\C_0)$. By the definition of $\cov(\M(\IR))=\cov(\M)$ there exists a cover $\U\subseteq\M(\IR)$ of the real line $\IR$ such that $|\U|=\cov(\M)$. The cover $$\U^n=\Big\{\prod_{k=1}^nC_k:C_1,\dots,C_n\in\U\Big\}$$ of $\IR^n$ has cardinality $|\U|^n=\cov(\M)$ and by Lemma~\ref{l4.3} is contained in the ideal $\sigma\C_0$, whence $\cov(\sigma\C_0)\le|\U^n|=\cov(\M)$.

\vspace{0.2 cm}

\noindent (2) $\non(\sigma\C_0)=\non(\M)$.

\vspace{0.2 cm}

The inequality $\non(\sigma\C_0)\le\non(\M)$ trivially follows from the inclusion $\sigma\C_0\subseteq\M$.

We will prove the inequality $\non(\sigma\C_0)\ge\non(\M)$. Let $A\subseteq\IR^n$ and $|A|<\non(\M)$.
It follows that $A\subseteq B^n$ for some set $B\subseteq\IR$ of cardinality $|B|\le n\cdot |A|<\non(\M)=\non(\M(\IR))$. Hence $B$ is meager. By Lemma~\ref{l4.3}, the power $B^n$ belongs to the ideal $\sigma\C_0$.

\vspace{0.2 cm}

\noindent (3) $\add(\sigma\C_0)=\add(\sigma\C_0,\M)=\add(\M)$.

\vspace{0.2 cm}

Since $\add(\sigma\C_0)\le\add(\sigma\C_0,\M)$, it suffices to prove two inequalities $\add(\sigma\C_0,\M)\le\add(\M)$ and $\add(\M)\le\add(\sigma\C_0)$.

First we prove the inequality $\add(\sigma\C_0,\M)\le\add(\M)$. Let $\A\subseteq \M(\IR)$  be a subfamily of cardinality $|\A|=\add(\M(\IR))=\add(\M)$ whose union $\bigcup\A$ is not meager in $\IR$. It follows that the family $\A^n=\{\prod_{k=1}^nA_k:A_1,\dots,A_n\in\A\}$ has cardinality $|\A|^n=\add(\M)$ and, by Lemma~\ref{l4.3}, is contained in the ideal $\sigma\C_0(\IR^n)=\sigma\C_0$. The Kuratowski-Ulam Theorem~\cite[8.41]{Ke} implies that the union $\bigcup\A^n=\left(\bigcup\A\right)^n$ is not meager in $\IR^n$. Hence $\add(\sigma\C_0,\M)\le|\A^n|=\add(\M)$.
\smallskip

The inequality $\add(\M)\le\add(\sigma\C_0)$ will follow if we show that for each family $\A$ containing less than $\add(\M)$ tame Cantor sets in $\IR^n$ the union $\bigcup\A$ belongs to the ideal $\sigma\C_0$. Let $G=(\IR\setminus\IQ)^n$. Of course, $G$ is a dense $G_\delta$-subset of $\IR^n$. By Lemma~\ref{l3.5}, for each tame Cantor set $A\in\A$ the set $\HH^G_A=\{h\in\HH(\IR^n):h(A)\subseteq G\}$ is a dense $G_\delta$-set in the homeomorphism group $\HH(\IR^n)$. Since $|\A|<\add(\M)\le\cov(\M)=\cov(\M(\HH(\IR^n)))$, the intersection $\bigcap_{A\in\A}\HH^G_A$ is not empty.
Let $h\in\bigcap_{A\in\A}\HH^G_A$. It follows that $h(\A)=\{h(A):A\in\A\}$ is a family of less than $\add(\M)$ many compact subsets of the space $G=(\IR\setminus \IQ)^n$, which is homeomorphic to the Baire space $\w^\w$. Since $|h(\A)|<\add(\M)\le \mathfrak b=\add(\sigma\K(G))$, there is a $\sigma$-compact subset $K\subseteq G$ containing the union $\bigcup_{A\in\A}h(A)$. Lemma~\ref{l4.4} guarantees that $K\in \sigma\C_0(\IR^n)$. Thus $h^{-1}(K)$ belongs to the ideal $\sigma\C_0$ and $\bigcup\A\subseteq h^{-1}(K)$.

\vspace{0.2 cm}

\noindent (4) $\cof(\sigma\C_0)=\cof(\sigma\C_0,\M)=\cof(\M)$.

\vspace{0.2 cm}

Since $\cof(\sigma\C_0,\M)\le\cof(\sigma\C_0)$ and $\cof(\M)=\max\{\non(\M),\mathfrak d\}$, it suffices to prove two inequalities:\newline $\max\{\non(\M),\mathfrak d\}\le\cof(\sigma\C_0,\M)$ and $\cof(\sigma\C_0)\le\max\{\non(\M),\mathfrak d\}$.

First we will prove the inequality $\max\{\non(\M),\mathfrak d\}\le\cof(\sigma\C_0,\M)$. In fact, we shall prove separately the inequalities
$\non(\M)\le \cof(\sigma\C_0,\M)$ and  $\mathfrak{d}\le \cof(\sigma\C_0,\M)$.

Let us prove the inequality $\non(\M)\le \cof(\sigma\C_0,\M)$. Let $\A\subseteq\M$ be a family of cardinality $|\A|=\cof(\sigma\C_0,\M)$ such that each set $C\in\sigma\C_0$ is contained in some set $A\in \A$. For each set $A\in\A$ we choose a point $x_A\in X\setminus A$. It follows that the set $B=\{x_A:A\in\A\}$ does not belong to the ideal $\sigma\C_0$. Hence $\non(\M)=\non(\sigma\C_0)\le|B|\le|\A|=\cof(\sigma\C_0,\M)$.

Now let us prove the inequality $\mathfrak{d}\le \cof(\sigma\C_0,\M)$.
Let $G=(\IR\setminus\IQ)^n$. The set $G$ is homeomorphic to the Baire space $\w^\w$.
By Lemma~\ref{add-cof}(3), $\cof(\sigma\K(G),\M(G))=\mathfrak d$.
By Lemma~\ref{l4.4}, $\sigma\K(G)\subseteq\sigma\C_0(\IR^n)$. Moreover, $\M(G)=\{G\cap M:M\in\M(\IR^n)\}$. Now we see that $\mathfrak d=\cof(\sigma\K(G),\M(G))\le\cof(\sigma\C_0,\M)$ and the proof of the inequality $\mathfrak{d}\le \cof(\sigma\C_0,\M)$ is completed.

Finally, we will prove the inequality $\cof(\sigma\C_0)\le \cof(\M)=\max\{\non(\M),\mathfrak d\}$. Since $\cof(\sigma\K(G))=\mathfrak d$, the ideal $\sigma\K(G)$ has a base $\mathcal D\subseteq\sigma\K(G)$ of cardinality $|\mathcal D|=\mathfrak d$.
By Lemma~\ref{l4.4} we have $\mathcal D\subseteq\sigma\C_0(\IR^n)$.

In the homeomorphism group $\HH(\IR^n)$ fix any non-meager subset $H$ of
cardinality $|H|=\non(\M(\HH(\IR^n)))=\non(\M)$. It is clear that the family
$\C=\{h^{-1}(D):h\in H,\;D\in\mathcal D\}$ has cardinality $|\C|\le|H\times\mathcal D|\le\max\{\non(\M),\mathfrak d\}=\cof(\M)$ and that $\C\subseteq \sigma\C_0$. We will complete the proof if we show that the family $\C$ is a base for the family $\sigma\C_0$. Let $A\in\sigma\C_0$. Without loss of generality we can assume that $A$ is $\sigma$-compact. By Lemma~\ref{l3.5}, the set $\HH^G_A=\{h\in\HH(\IR^n):h(A)\subseteq G\}$ is a dense $G_\delta$-set in $\HH(\IR^n)$ and hence it meets the non-meager set $H$. Consequently, there is a homeomorphism $h\in H$ such that $h(A)\subseteq G$. Because $\mathcal D$ is a base for $\sigma\K(G)$, the $\sigma$-compact set $h(A)$ is contained in some $\sigma$-compact set $D\in\mathcal D$. Then $A\subseteq h^{-1}(D)\in \C$ and the proof is complete. \hfill $\Box$

\section{Proof of Theorem~\ref{tD}}\label{s:tD}

For the proof of Theorem~\ref{tD} we shall use a deep result of Geoghegan and Summerhill \cite{GS} on the existence of $k$-dimensional pseudoboundaries in Euclidean spaces.

A subset $\Sigma\subset\IR^n$ is defined to be a {\em $\mathcal B$-pseudoboundary} for a topologically invariant family $\mathcal B$ of closed subsets of $\IR^n$ if
\begin{itemize}
\item $\Sigma=\bigcup_{n\in\w}B_n$ for some sets $B_n\in\mathcal B$, $n\in\w$;
\item for each $B\in\mathcal B$, an open set $U\subset\IR^n$ and an open cover $\V$ of $U$ there is a homeomorphism $h:U\to U$ such that $h(U\cap B)\subset \Sigma$ and for each point $x\in U$ the doubleton $\{x,h(x)\}$ is contained in some set $V\in \V$.
\end{itemize}

Geoghegan and Summerhill \cite{GS} for every numbers $0\le k<n$ constructed a $\Z_{n-k-1}^*$-pseudoboundary $\Sigma$ for the family $\Z^*_{n-k-1}$ of strong $Z_{n-k-1}$-sets in $\IR^n$.

A closed subset $A\subset\IR^n$ is called a {\em strong $Z_k$-set in $\IR^n$} if for any compact subpolyhedron $P\subset\IR^n$ of dimension $\dim(P)\le k$ and any $\e>0$ there is an $\e$-isotopy $(h_t)_{t\in[0,1]}:\IR^n\to\IR^n$ such that $h_0$ is the identity homeomorphism of $\IR^n$, $h_1(A)\cap P=\emptyset$, and $h_t(x)=x$ for any $t\in[0,1]$ and any point $x\in \IR^n$ on distance $\mathrm{dist}(x,A\cap P)\ge \e$ from $A\cap P$. The following lemma was proved in \cite{GS},  Proposition 3.1(4). In this lemma $\Z_2$ and $\Z^*_{n-k-1}$ stand for the families of $Z_2$-sets and strong $Z_{n-k-1}$-sets in $\IR^n$, respectively.

\begin{lemma}\label{l6.1} For every $n\ge 5$ and $0\le k\le n-3$ we have the inclusions
$$\oD_k\cap\Z_2\subset \Z^*_{n-k-1}\subset\oD_k.$$
\end{lemma}

We shall use the following non-trivial result of Geoghegan and Summerhill \cite[3.12]{GS}.

\begin{lemma}\label{l6.2} For every $n\in\IN$ and $0\le k<n$ there is a $\Z^*_{n-k-1}$-pseudoboundary $\Sigma\subset\IR^n$.
\end{lemma}

With Lemmas~\ref{l6.1} and \ref{l6.2} in our disposition we now are able to prove Theorem~\ref{tD}. So, fix numbers $n\in\IN$ and $k<n$, and consider the $\sigma$-ideal $\sigma\oD_k$ generated by closed $k$-dimensional sets in the Euclidean space $\IR^n$. By Corollary~\ref{any-ideal},
$$\add(\sigma\oD_k)\le\add(\M),\;\;\cov(\sigma\oD_k)=\cov(\M),\;\;\non(\sigma\oD_k)\le\non(\M) \mbox{ and }\cof(\sigma\oD_k)\ge\cof(\M).$$
So, it remains to check that $\add(\sigma\oD_k)\ge\add(\M)$ and $\cof(\sigma\oD_k)\le\cof(\M)$.
Identify the Euclidean space $\IR^n$ with a linear subspace of the Euclidean space $\IR^m$ for some $m\ge n+4\ge 5$. Then $\IR^n$ is a $Z_2$-set in $\IR^m$, and Lemma~\ref{l6.1} guarantees that $\oD_k\cap \Z_2\subset\Z^*_{m-k-1}$, where $\Z_2$ and $\Z^*_{m-k-1}$ stand for the families of $Z_2$-sets and strong $Z^*_{m-k-1}$-sets in $\IR^m$, respectively. By Lemma~\ref{l6.2}, the Euclidean space $\IR^m$ contains a $\Z^*_{m-k-1}$-pseudoboundary $\Sigma$. The set $\Sigma$, being the countable union of closed $k$-dimensional subsets, is $k$-dimensional and hence can be enlarged to a dense $k$-dimensional $G_\delta$-subset $G\subset\IR^m$, see Theorem 1.5.11 in \cite{En2}. Since $\dim(G)=k<m$, the Baire Theorem implies that the set $G$ is not $\sigma$-compact and hence $G$ is the image of the Baire space $\w^\w$ under a perfect map. This fact can be used to prove that $\add(\sigma\K(G))=\add(\sigma\K)=\mathfrak b$ and $\cof(\sigma\K(G))=\cof(\sigma\K)=\mathfrak d$.
\smallskip

(1) To prove that $\add(\sigma\oD_k)\ge\add(\M)$, fix any subfamily $\A\subset\sigma\oD_k$ of cardinality $|\A|<\add(\M)$. Since each set $A\in\A$ is contained in a countable union of compact $k$-dimensional subsets, we lose no generality assuming that each set $A\in\A$ is compact.

The definition of the $\Z^*_{m-k-1}$-pseudoboundary guarantees that for each set $A\in\A\subset \Z^*_{m-k-1}$ the set $\HH_A^\Sigma=\{h\in\HH(\IR^m):h(A)\subset\Sigma\}$ is dense in the homeomorphism group $\HH(\IR^m)$ and hence the $G_\delta$-subset
$$\HH_A^G=\{h\in\HH(\IR^m):h(A)\subset G\}\supset\HH_A^\Sigma$$ also is dense in $\HH(\IR^m)$.
Since $|\A|<\add(\M)\le\cov(\M)$, the intersection $\bigcap_{A\in\A}\HH_A^G$ contains some homeomorphism $h:\IR^m\to\IR^m$. This homeomorphism maps the union $\bigcup\A$ into $G$. Consequently, $h(\A)=\{h(A):A\in\A\}\subset\sigma\K(G)$. Since $|h(\A)|=|\A|<\add(\M)\le\mathfrak b=\add(\sigma\K(G))$, we conclude that the union $h(\cup\A)$ is contained in some $\sigma$-compact subset $K\subset G$. Then the set $A=\IR^n\cap h^{-1}(K)$ is a $\sigma$-compact subset of $\IR^n$ of dimension $\dim(A)\le\dim(h^{-1}(K))=\dim(K)\le\dim (G)=k$ containing the union $\bigcup\A$ and witnessing that $\add(\sigma\oD_k)\ge\add(\M)$.
\smallskip

(2) Next, we prove that $\cof(\sigma\oD_k)\le\cof(\M)=\max\{\non(\M),\mathfrak d\}$. Since $\sigma\K(G)=\mathfrak d$, the ideal $\sigma\K(G)$ has a base $\mathcal D\subseteq\sigma\K(G)$ of cardinality $|\mathcal D|=\mathfrak d$.
In the homeomorphism group $\HH(\IR^m)$ fix any non-meager subset $H$ of
cardinality $|H|=\non(\M(\HH(\IR^m)))=\non(\M)$. It is clear that the family
$\C=\{\IR^n\cap h^{-1}(D):h\in H,\;D\in\mathcal D\}$ has cardinality $|\C|\le|H\times\mathcal D|\le\max\{\non(\M),\mathfrak d\}=\cof(\M)$ and that $\C\subseteq \sigma\oD_k$. We will complete the proof if we show that the family $\C$ is a base for the ideal $\sigma\oD_k$.

Let $A\in\sigma\oD_k$. Without loss of generality we can assume that $A$ is $\sigma$-compact and hence can be written as the countable union $A=\bigcup_{i\in\w}A_i$ of compact subsets $A_i\in\oD_k$, $i\in\w$. Since $\Sigma$ is a $\Z^*_{m-k-1}$-pseudoboundary in $\IR^m$ and  $A_i\in \oD_k\cap\Z_2\subset\Z^*_{m-k-1}$ for $i\in\w$, the set $$\HH^G_A=\{h\in\HH(\IR^m):h(A)\subseteq G\}=\bigcap_{i\in\w}\HH^G_{A_i}$$ is a dense $G_\delta$-set in $\HH(\IR^m)$ and hence it meets the non-meager set $H$. Consequently, there is a homeomorphism $h\in H$ such that $h(A)\subseteq G$. Because $\mathcal D$ is a base for $\sigma\K(G)$, the $\sigma$-compact set $h(A)$ is contained in some $\sigma$-compact set $D\in\mathcal D$. Then $A\subseteq \IR^n\cap h^{-1}(D)\in \C$ and the proof is complete. \hfill $\Box$


\begin{thebibliography}{99}

\bibitem{An} L.~Antoine, {\em Sur l'homeomorphie de deux figures et de leurs voisinages}, J. Math. Pures Appl. {\bf 86} (1921), 221--325.

\bibitem{hilbert} T.~Banakh, M.~Morayne, R.~Ra{\l}owski, S.~{\.Z}eberski, {\em Topologically invariant $\sigma$-ideals in the Hilbert cube}, preprint.

\bibitem{bartoszynskijudah} T.~Bartoszy\' nski, H.~Judah, {\em Set theory: on the structure of the real line}, Wellesley, MA: A. K. Peters Ltd., 1995.

\bibitem{Bing} R. H. Bing, {\em Tame Cantor sets in $E^3$}, Pacific J. Math. {\bf 11}:2 (1961), 435--446.

\bibitem{Blank} W.A.~Blankenship, {\em Generalization of a construction of Antoine}, Ann. of Math. (2) {\bf 53} (1951), 276--297.

\bibitem{Chi} A.~Chigogidze, {\em Inverse Spectra}, North-Holland Publishing Co., Amsterdam, 1996.

\bibitem{CW} K. Ciesielski, J. Wojciechowski, {\em Sums of connectivity functions on $R^n$}, Proc. London. Math. Soc. {\bf 76} (1998), 406--416.

\bibitem{En2} R.~Engelking, {\em Theory of dimensions finite and infinite}, Sigma Series in Pure Mathematics, 10. Heldermann Verlag, Lemgo, 1995.

\bibitem{fremlin} D. Fremlin, {\em Measure-additive coverings and measurable selectors}, Dissert. Math. {\bf 260} (1987).

\bibitem{GS} R.~Geoghegan, R.~Summerhill, {\em Pseudo-boundaries and pseudo-interiors in Euclidean spaces and topological manifolds}, Trans. Amer. Math. Soc. {\bf 194} (1974), 141--165.

\bibitem{Ke} A. S. Kechris, {\em Classical descriptive set theory}, Graduate Texts in Mathematics, 156. Springer-Verlag, New York, 1995.

\bibitem{kuratowski} K. Kuratowski, {\em Topology. I},  Academic Press, New York-London, Pañstwowe Wydawnictwo Naukowe, Warsaw, 1966.

\bibitem{Menger} K.~Menger, {\em Allgemeine Raume und Cartesische Raume Zweite Mitteilung: ``Uber umfas - sendste $n$-dimensional Mengen''}, Proc. Acad. Amsterdam {\bf 29} (1926), 1125--1128.

\bibitem{McM} D.R. Jr. McMillan, {\em Taming Cantor sets in $E^n$}, Bull. Amer. Math. Soc. {\bf 70} (1964), 706--708.

\bibitem{Osborn} R.P.~Osborne, {\em Embedding Cantor sets in a manifold}, Michigan Math. J. {\bf 13} (1966), 57--63.

\bibitem{Sher} R.B.~Sher, {\em Concerning wild Cantor sets in $E^3$}, Proc. Amer. Math. Soc. {\bf 19} (1968), 1195--1200.

\end{thebibliography}
\end{document}